\documentclass[12pt]{article}

\usepackage[T1]{fontenc}
\usepackage[utf8]{inputenc}
\usepackage{lmodern}

\usepackage[left=1.05in,right=1.05in,top=1in,bottom=1in]{geometry}
\usepackage{amsmath,amssymb,amsfonts,amsthm,mathtools}
\usepackage{bm}
\usepackage{graphicx}
\usepackage{float}
\usepackage{tikz}
\usetikzlibrary{arrows.meta,calc,positioning,decorations.markings}
\usepackage{enumitem}
\usepackage{xcolor}
\usepackage{microtype}
\usepackage{hyperref}
\usepackage[noabbrev]{cleveref}

\hypersetup{
  colorlinks=true,
  linkcolor=blue!50!black,
  citecolor=blue!60!black,
  urlcolor=blue!60!black
}

\numberwithin{equation}{section}

\newtheorem{theorem}{Theorem}[section]
\newtheorem{corollary}[theorem]{Corollary}

\newtheorem{lemma}[theorem]{Lemma}
\theoremstyle{definition}
\newtheorem{definition}[theorem]{Definition}
\theoremstyle{remark}
\newtheorem{remark}[theorem]{Remark}

\newcommand{\arxivgraphic}[2]{%
  \IfFileExists{#2}{\includegraphics[#1]{#2}}%
  {\fbox{\parbox[c][4cm][c]{0.80\linewidth}{%
  \centering Figure file \texttt{#2} not included.}}}%
}

\title{Schwarz-Type Null Curves in $\mathbb{C}^4$: Symmetry and Period Reduction}

 \author{Erhan G\"uler \and Magdalena Toda }

\date{}

\begin{document}

\maketitle

\begin{abstract}
We study a genus-three hyperelliptic holomorphic null-curve family in
$\mathbb{C}^4$, modelled on the algebraic data of the classical Schwarz
P/D family, whose real parts define minimal immersions into
$\mathbb{R}^4$. For the order-four automorphism
$\omega\mapsto i\omega$ of the underlying Schwarz curve, we compute
explicitly its action on the four holomorphic Weierstrass $1$-forms and
derive the resulting identities for all real period vectors.
Consequently, for every lattice invariant under the induced target
rotation, torus-period closure can be checked on one representative
from each orbit of a symmetry-stable homology generating set. We also
identify precisely when the additional parameter produces a
nondegenerate codimension-two deformation. The paper does not claim
the construction of a new embedded periodic minimal surface in
$\mathbb{R}^4$; rather, it provides an explicit symmetry reduction of
the period problem associated with this family.
\end{abstract}



\tableofcontents
\bigskip

\section{Introduction}

Minimal surfaces have long occupied a central place in differential
geometry, with their development shaped by the interaction between
complex analysis, geometry, and symmetry. Classical contributions of
Schwarz \cite{Schwarz}, together with the later treatments of Osserman
\cite{Osserman} and Nitsche \cite{Nitsche}, established much of the
analytic and geometric framework underlying the modern theory.

In Euclidean three-space, the theory has developed considerably.
The works of Hoffman and Meeks \cite{HoffmanMeeks}, as well as
López and Ros \cite{LopezRos}, demonstrate that minimal surfaces
exhibit a rich global structure even under strong topological
constraints. Periodic and triply periodic minimal surfaces provide
another important direction. Grosse-Brauckmann
\cite{GBrauckmannSurvey} and Weyhaupt \cite{WeyhauptGyroid}, among
others, studied such surfaces and their geometric structures,
highlighting the importance of symmetry and period-closing conditions.

Two classical sources of inspiration are the minimal surfaces associated
with Riemann and Schwarz. Riemann's minimal examples form a family of
singly periodic minimal surfaces whose Weierstrass data are closely
related to elliptic functions. Schwarz's constructions, on the other
hand, provide highly symmetric triply periodic minimal surfaces,
including the classical P and D surfaces and their associate family.
Their construction illustrates the fundamental role played by discrete
symmetries, fundamental domains, and period closure in producing
globally periodic minimal surfaces.

In higher codimension, new geometric and analytic phenomena appear.
Grosse-Brauckmann and Kürsten \cite{GrosseBrauckmannKuersten}
constructed embedded $n$-periodic minimal surfaces in $\mathbb{R}^n$
by solving Plateau problems for cubical Jordan curves and extending
the resulting minimal disks by Schwarz reflections. Their construction
gives, in particular, embedded periodic examples in $\mathbb{R}^4$.
The approach of the present paper is different. Rather than using a
reflection--Plateau construction, we study the period map associated
with an explicit holomorphic null-curve family defined on a fixed
genus-three hyperelliptic curve.

Minimal surfaces in $\mathbb{R}^4$ exhibit features that have no direct
counterpart in the classical codimension-one setting. In particular,
their normal bundle has rank two, and their Weierstrass data possess
additional degrees of freedom. A conformal minimal immersion into
$\mathbb{R}^4$ can be described locally as the real part of an integral
of holomorphic $1$-forms
\[
\Phi=(\phi_1,\phi_2,\phi_3,\phi_4)
\]
satisfying the nullity condition
\[
\phi_1^2+\phi_2^2+\phi_3^2+\phi_4^2=0.
\]
Thus, the underlying holomorphic null curve naturally takes values in
$\mathbb{C}^4$, whereas its real part defines a minimal immersion into
$\mathbb{R}^4$. This distinction is particularly important when period
conditions and lattice quotients are considered.

Several works have investigated minimal surfaces in higher-dimensional
Euclidean spaces and flat tori. Small \cite{SmallR4} studied algebraic
minimal surfaces in $\mathbb{R}^4$, while Shoda
\cite{ShodaJLMS,ShodaTrigonal,ShodaModuliArXiv} investigated moduli
spaces and constructed families of minimal surfaces in flat tori.
Pirola \cite{PirolaSpin} studied related variational questions involving
spin structures and periodic minimal surfaces. More recently, Toda and
G\"uler \cite{TodaGulerR4}, and G\"uler and Toda \cite{TodaGulerRn},
developed generalized Weierstrass-type representations and explicit
constructions for minimal surfaces in $\mathbb{R}^4$ and
$\mathbb{R}^n$.

The interaction between minimal surface theory and other areas of
geometry has also produced important developments. Bolton, Pedit, and
Woodward \cite{BoltonPeditWoodward} established connections with affine
Toda field equations, while Guest \cite{Guest} developed a broader
framework relating harmonic maps and integrable systems. Representation
formulae in non-Euclidean settings were studied, for example, by
Aiyama and Akutagawa \cite{AiyamaAkutagawa}. Classical references on
complex geometry and special functions, such as Kobayashi
\cite{Kobayashi} and Byrd and Friedman \cite{ByrdFriedman}, provide
additional analytic tools relevant to these constructions.

The purpose of the present paper is to reinterpret a specific aspect of
the classical Schwarz construction within the framework of holomorphic
null curves in $\mathbb{C}^4$ and minimal immersions into
$\mathbb{R}^4$. Starting from algebraic data modelled on the classical
Schwarz P/D family, we consider a genus-three hyperelliptic
holomorphic null-curve family and study its associated real period map.

A central role is played by the order-four automorphism
\[
\omega\longmapsto i\omega
\]
of the underlying Schwarz curve. We compute explicitly the induced
action of this automorphism on the four holomorphic Weierstrass
$1$-forms. The resulting equivariance relations yield concrete linear
identities among the corresponding real period vectors. Consequently,
for a lattice invariant under the induced target rotation, the
torus-period condition can be reduced to one representative from each
orbit of a symmetry-stable homology generating set.

We also determine precisely when the additional parameter in the
holomorphic data produces a nondegenerate codimension-two deformation.
Thus, the principal contribution of the paper is not the construction
of a new embedded periodic minimal surface, but an explicit
symmetry reduction of the period problem for this Schwarz-type
holomorphic null-curve family.

Accordingly, the term ``Schwarz-type'' is used here in a precise and
restricted sense. It refers to the common algebraic seed and to the
symmetry-driven treatment of the period map. We do not claim that the
classical embedded Schwarz P, D, or related surfaces have been lifted
to embedded periodic minimal surfaces in $\mathbb{R}^4$, nor do we
establish complete period closure or embeddedness for the family
considered here.

\subsection*{Organization and scope}

Section~\ref{sec:R4WE} recalls the Weierstrass representation of
minimal surfaces in $\mathbb{R}^4$ in terms of holomorphic null curves
in $\mathbb{C}^4$.
Section~\ref{sec:seed} introduces the Schwarz-inspired genus-three
hyperelliptic seed and derives the corresponding null $1$-forms.
Section~\ref{sec:elliptic} discusses the relevant hyperelliptic
primitives and their reduction to elliptic integrals.
Section~\ref{sec:PDG} reviews the classical Schwarz P/D/associate-family
picture and formulates the corresponding codimension-two deformation.
Section~\ref{sec:period} introduces the real period map, periodicity in
flat four-tori, and the associated lattice-period conditions.
Sections~\ref{sec:symmetry} and \ref{sec:strategy} develop the symmetry
and homological framework used to reduce these conditions.
Section~\ref{sec:newresult} contains the main symmetry-reduction
theorem. The final section discusses the scope and limitations of the
construction and distinguishes it from known reflection--Plateau
constructions. For completeness, Appendix~A records standard integral
normal forms, while Appendix~B discusses the underlying hyperelliptic
structure and its relevant symmetries.

% -------------------------
% Schwarz-Type Minimal Surfaces in R^4
% -------------------------
\section{Schwarz-Type Minimal Surfaces in $\mathbb{R}^4$}\label{sec:R4WE}

In this section we set up Schwarz-type minimal surfaces in $\mathbb{R}^4$
within the adaptive Weierstrass framework developed in \cite{TodaGulerRn} and \cite{TodaGulerR4}.
Our approach is compatible with the null-holomorphic representation and emphasizes symmetry of the associated holomorphic data.

\subsection{Weierstrass data in $\mathbb{R}^4$}

Let $D\subset\mathbb{C}$ be a simply connected domain.
Following \cite{TodaGulerRn} and \cite{TodaGulerR4}, conformal minimal immersions into $\mathbb{R}^4$ arise from holomorphic data
\[
f,\ g_1,\ g_2 : D \longrightarrow \mathbb{C}
\]
via the coefficient functions
\begin{equation}\label{eq:WE-R4}
\begin{aligned}
\phi_1 &= \frac12 f(1 - g_1^2 - g_2^2),\\
\phi_2 &= \frac{i}{2} f(1 + g_1^2 + g_2^2),\\
\phi_3 &= f g_1,\\
\phi_4 &= f g_2 ,
\end{aligned}
\end{equation}
and the associated $\mathbb{C}^4$-valued holomorphic $1$-form
\[
\Phi = (\Phi_1,\Phi_2,\Phi_3,\Phi_4):=(\phi_1,\phi_2,\phi_3,\phi_4)\,d\omega.
\]
These data satisfy identically the nullity condition (in coefficient form)
\[
\phi_1^2+\phi_2^2+\phi_3^2+\phi_4^2=0,
\]
equivalently $\Phi_1^2+\Phi_2^2+\Phi_3^2+\Phi_4^2=0$ as $1$-forms.

The associated minimal immersion is
\[
X(\omega)=\Re\int_{\omega_0}^{\omega}\Phi
=\Re\int_{\omega_0}^{\omega}(\phi_1,\phi_2,\phi_3,\phi_4)\,d\omega,
\]
defined wherever the induced metric does not vanish.

\begin{remark}
On a simply connected domain, the integral defining $X$ is path independent.
In codimension one, the Weierstrass-Enneper representation in $\mathbb{R}^3$ is often written in terms of a meromorphic Gauss map and a holomorphic $1$-form.
Formula \eqref{eq:WE-R4} is a codimension-two analogue: $(g_1,g_2)$ encodes two complex degrees of freedom, and the single quadratic constraint is exactly the holomorphic nullity condition.
\end{remark}

\begin{remark}[Induced metric: coordinate-invariant viewpoint]
The induced metric is
\[
ds^2=|\Phi_1|^2+|\Phi_2|^2+|\Phi_3|^2+|\Phi_4|^2,
\]
where $|\Phi_j|$ denotes the pointwise norm of a $1$-form in any local complex coordinate.
In a local coordinate $\omega$ on which $\Phi_j=\phi_j(\omega)\,d\omega$, one computes
\[
ds^2=\bigl(|\phi_1|^2+|\phi_2|^2+|\phi_3|^2+|\phi_4|^2\bigr)\,|d\omega|^2
=\frac{|f|^2}{2}\bigl(1+|g_1|^2+|g_2|^2\bigr)^2|d\omega|^2.
\]

Regularity holds precisely where $\Phi$ does not vanish; in particular, this excludes points where $f=0$ or where the coefficient vector in \eqref{eq:WE-R4} vanishes.

When the data are pulled back to a branched cover, one must interpret this expression in a genuine local coordinate on the cover (see Appendix~B for the Schwarz seed, where apparent singularities in the $\omega$-projection cancel in a local coordinate on $\Sigma$).
\end{remark}

% -------------------------
% Schwarz seed
% -------------------------
\section{A Schwarz-Type Holomorphic Seed}\label{sec:seed}

\subsection{Schwarz-type holomorphic seed}

To produce Schwarz-type geometry, we choose a holomorphic seed whose algebraic structure mirrors the classical Schwarz polynomial.
Let
\begin{equation}\label{eq:Schwarz-seed}
f(\omega)=\frac{2}{\sqrt{\omega^8-14\omega^4+1}},
\end{equation}
where the square root is taken on the two-sheeted branched covering
\[
\Sigma=\{(\omega,y)\in\mathbb{C}^2 : y^2=\omega^8-14\omega^4+1\}.
\]

The zeros of the polynomial
\[
P(\omega)=\omega^8-14\omega^4+1
\]
are precisely the branch points of $\Sigma\to\mathbb{C}$.
Setting $t=\omega^4$, the equation $P(\omega)=0$ reduces to
\[
t^2-14t+1=0,
\qquad
t=7\pm4\sqrt{3}.
\]
Since
\[
7\pm4\sqrt{3}=(2\pm\sqrt{3})^2,
\]
we obtain
\[
\omega^4=(2\pm\sqrt{3})^2,
\]
and hence the eight branch points are
\[
\omega\in
\Big\{
\pm\sqrt{2+\sqrt{3}},
\ \pm i\sqrt{2+\sqrt{3}},
\ \pm\sqrt{2-\sqrt{3}},
\ \pm i\sqrt{2-\sqrt{3}}
\Big\}.
\]

\begin{remark}
Since $\deg P=8$ and $P$ has simple roots, the compactification of $\Sigma$ is a hyperelliptic curve of genus
$g=(8-2)/2=3$.
This genus matches the genus of the standard compact quotients of the Schwarz P/D surfaces in $\mathbb{T}^3$, and is one reason this algebraic seed is well aligned with the TPMS period picture.
\end{remark}

\subsection{Explicit null $1$-forms from the seed}

We now write the Weierstrass data explicitly with respect to the projection coordinate $\omega$.
Let
\[
f(\omega)=\frac{2}{\sqrt{\omega^8-14\omega^4+1}}=\frac{2}{y}, \qquad
g_1(\omega)=\omega, \qquad
g_2(\omega)=\lambda\omega,
\]
where $\lambda\in\mathbb{C}$ is a deformation parameter in the additional normal direction.
Since
\[
g_1^2(\omega)+g_2^2(\omega)=(1+\lambda^2)\omega^2,
\]
the coefficient functions in \eqref{eq:WE-R4} become
\[
\begin{aligned}
\phi_1(\omega)
&=
\frac{1-(1+\lambda^2)\omega^2}
{\sqrt{\omega^8-14\omega^4+1}},
\\[2mm]
\phi_2(\omega)
&=
\frac{i\bigl(1+(1+\lambda^2)\omega^2\bigr)}
{\sqrt{\omega^8-14\omega^4+1}},
\\[2mm]
\phi_3(\omega)
&=
\frac{2\omega}
{\sqrt{\omega^8-14\omega^4+1}},
\\[2mm]
\phi_4(\omega)
&=
\frac{2\lambda\omega}
{\sqrt{\omega^8-14\omega^4+1}}.
\end{aligned}
\]
Equivalently, the $\mathbb{C}^4$-valued $1$-form on $\Sigma$ is
\[
\Phi=(\phi_1,\phi_2,\phi_3,\phi_4)\,d\omega,
\]
and the nullity condition $\phi_1^2+\phi_2^2+\phi_3^2+\phi_4^2=0$ is identically satisfied.

\medskip

The associated minimal immersion is given locally by
\[
X(\omega)
=
\Re\int_{\omega_0}^{\omega}
\Phi
=
\Re\int_{\omega_0}^{\omega}
\bigl(\phi_1,\phi_2,\phi_3,\phi_4\bigr)\,d\omega,
\]
where the integral is taken on the hyperelliptic Riemann surface

\[
\Sigma
=
\bigl\{(\omega,y)\in\mathbb{C}^2:\ y^2=\omega^8-14\omega^4+1\bigr\}.
\]

\paragraph{Primitive functions.}
The coordinate functions of the immersion are given by the real parts of the primitives
\begin{equation}
\begin{aligned}
\int \Phi_1
&=
\int
\frac{1-(1+\lambda^2)\omega^2}
{\sqrt{\omega^8-14\omega^4+1}}
\,d\omega,
\\[2mm]
\int \Phi_2
&=
i\int
\frac{1+(1+\lambda^2)\omega^2}
{\sqrt{\omega^8-14\omega^4+1}}
\,d\omega,
\\[2mm]
\int \Phi_3
&=
2\int
\frac{\omega}
{\sqrt{\omega^8-14\omega^4+1}}
\,d\omega,
\\[2mm]
\int \Phi_4
&=
2\int
\frac{\lambda\omega}
{\sqrt{\omega^8-14\omega^4+1}}
\,d\omega.
\end{aligned}
\label{eq:weierstrass-Schwarz}
\end{equation}
Figures \ref{fig:x1x2x3-x1x2x4-proj1}--\ref{fig:polar-proj-2} present the projections of the Schwarz-type minimal surface defined in \eqref{eq:weierstrass-Schwarz}. 
They demonstrate the geometry of the immersion in $\mathbb{R}^4$ under both $(u,v)$ and $(r,\theta)$ parametrizations.

\paragraph{Structural remarks.}
All coordinate functions arise from hyperelliptic integrals on $\Sigma$, whose branch points are the eight zeros of $\omega^8-14\omega^4+1$.
Moreover, at the level of $1$-forms one has $\Phi_4=\lambda\,\Phi_3$, hence
\[
\int \Phi_4
=
\lambda\int \Phi_3,
\]
which induces an affine dependence between the corresponding components of the minimal immersion after taking real parts.

% -------------------------
% Elliptic reduction
% -------------------------
\section{Hyperelliptic Primitives and Reduction to Elliptic Integrals}\label{sec:elliptic}

\subsection{Reduction by the substitution $t=\omega^2$}

We now analyze the primitives above.
A key observation is that the relevant hyperelliptic integrals reduce to elliptic integrals after the substitution
\[
t=\omega^2,
\qquad
dt=2\omega\,d\omega.
\]
This reduction is classical for quartic radicals and underlies the elliptic character of the Schwarz P/D family.

\begin{figure}[t!]
    \centering
    \arxivgraphic{width=1\linewidth}{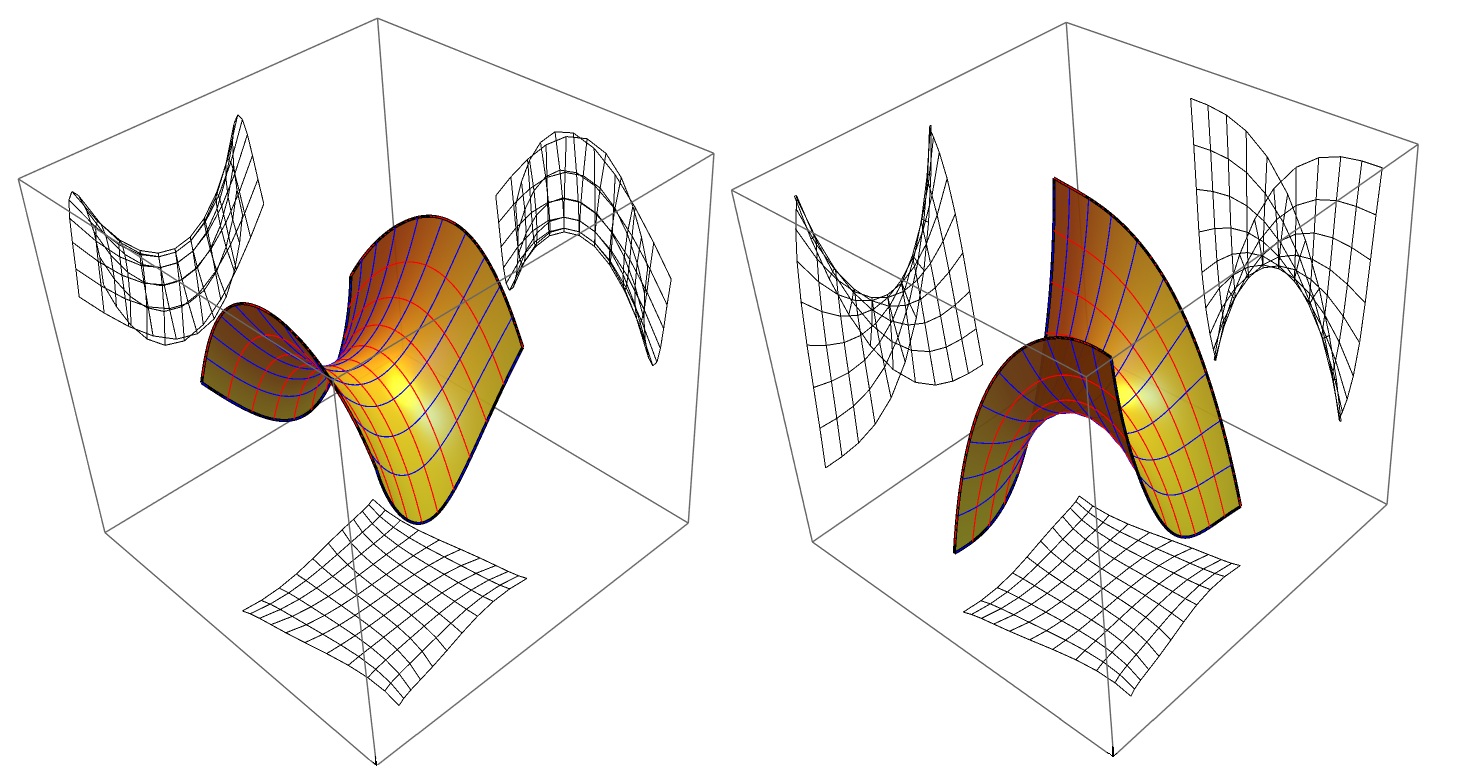}
   \caption{Projections of the Schwarz-type minimal surface $X(u,v)$,
defined in \eqref{eq:weierstrass-Schwarz}, into the $X_1X_2X_3$-space (left)
and the $X_1X_2X_4$-space (right).}
    \label{fig:x1x2x3-x1x2x4-proj1}
\end{figure}

\begin{figure}[t!]
    \centering
    \arxivgraphic{width=1\linewidth}{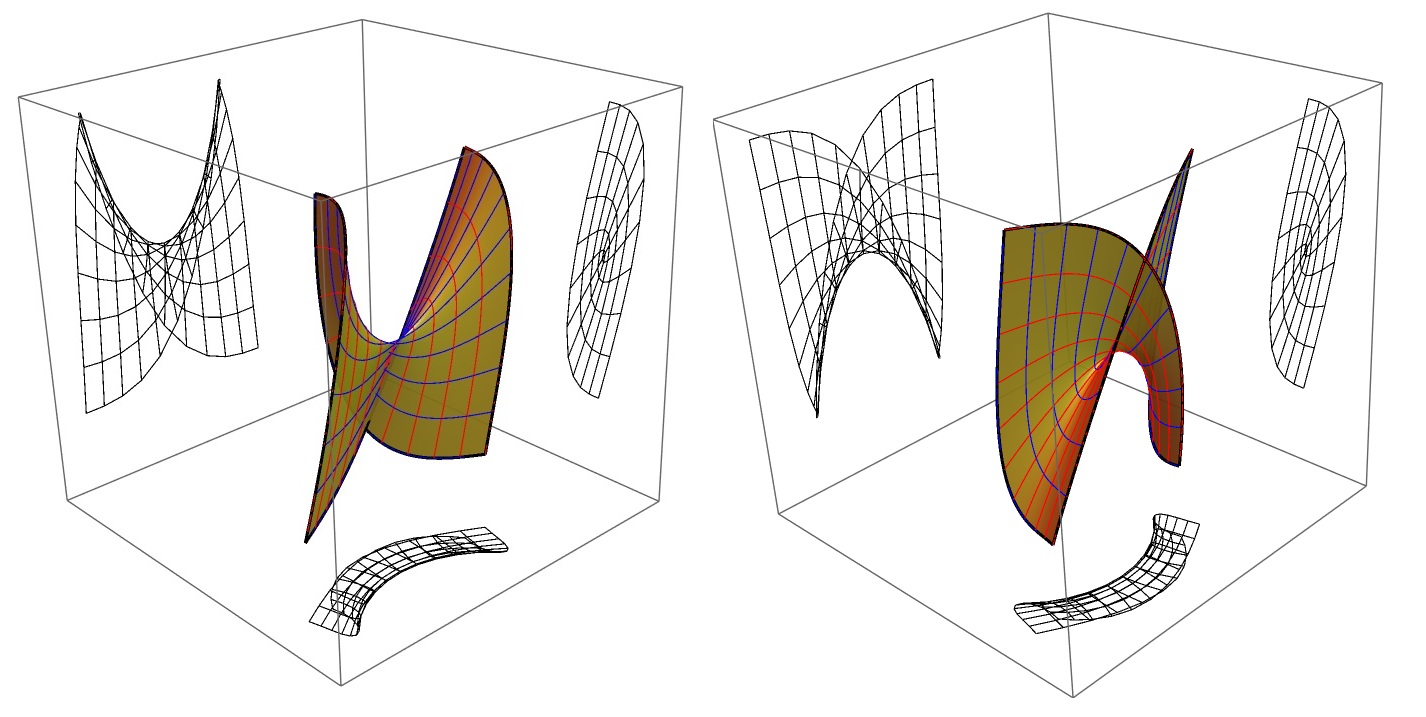}
   \caption{Projections of the Schwarz-type minimal surface $X(u,v)$,
defined in \eqref{eq:weierstrass-Schwarz}, into the $X_1X_3X_4$-space (left)
and the $X_2X_3X_4$-space (right).}
    \label{fig:x1x3x4-x2x3x4-proj2}
\end{figure}

\begin{figure}[t!]
    \centering
    \arxivgraphic{width=1\linewidth}{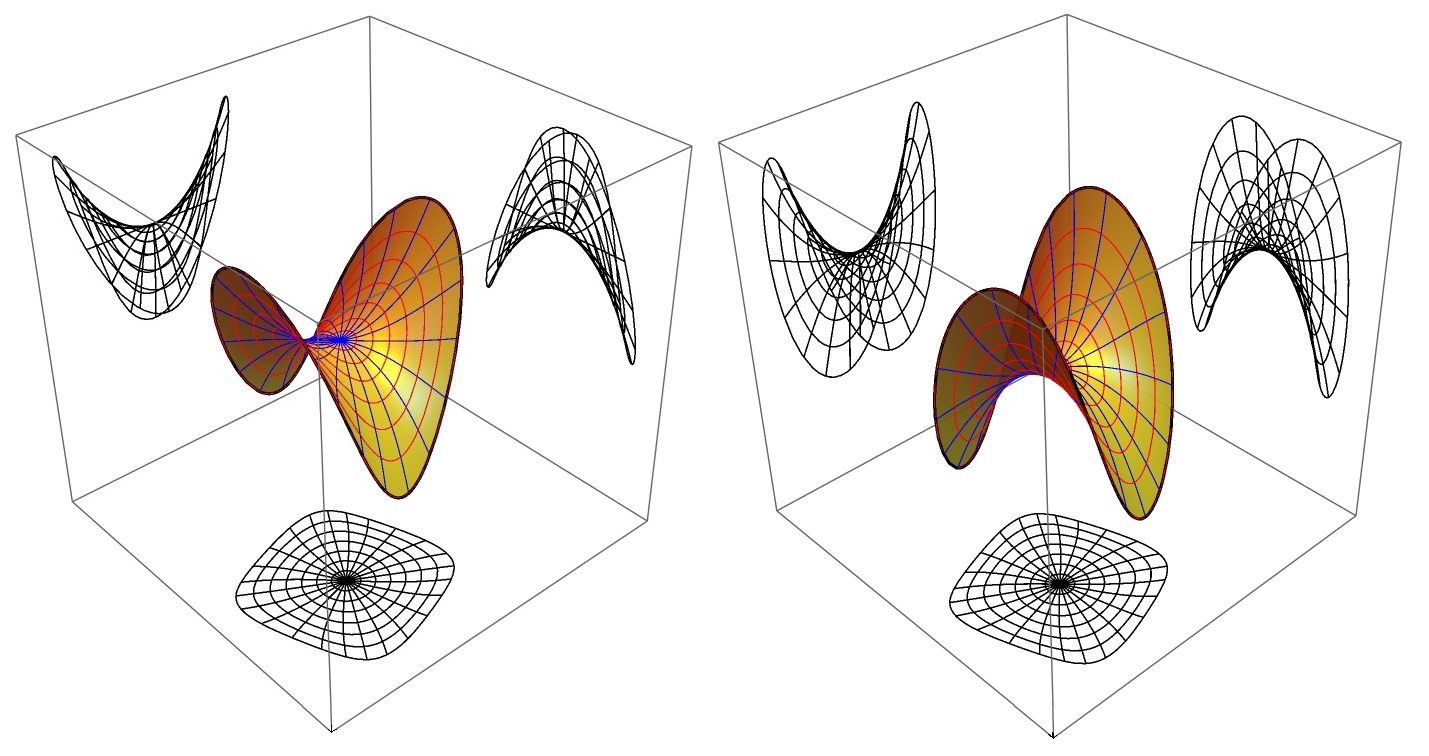}
   \caption{Projections of the Schwarz-type minimal surface $X(r,\theta)$ defined in \eqref{eq:weierstrass-Schwarz}, into the $X_1X_2X_3$-space (left)
and the $X_1X_2X_4$-space (right).}
    \label{fig:polar-proj-1}
\end{figure}

\begin{figure}[t!]
    \centering
    \arxivgraphic{width=1\linewidth}{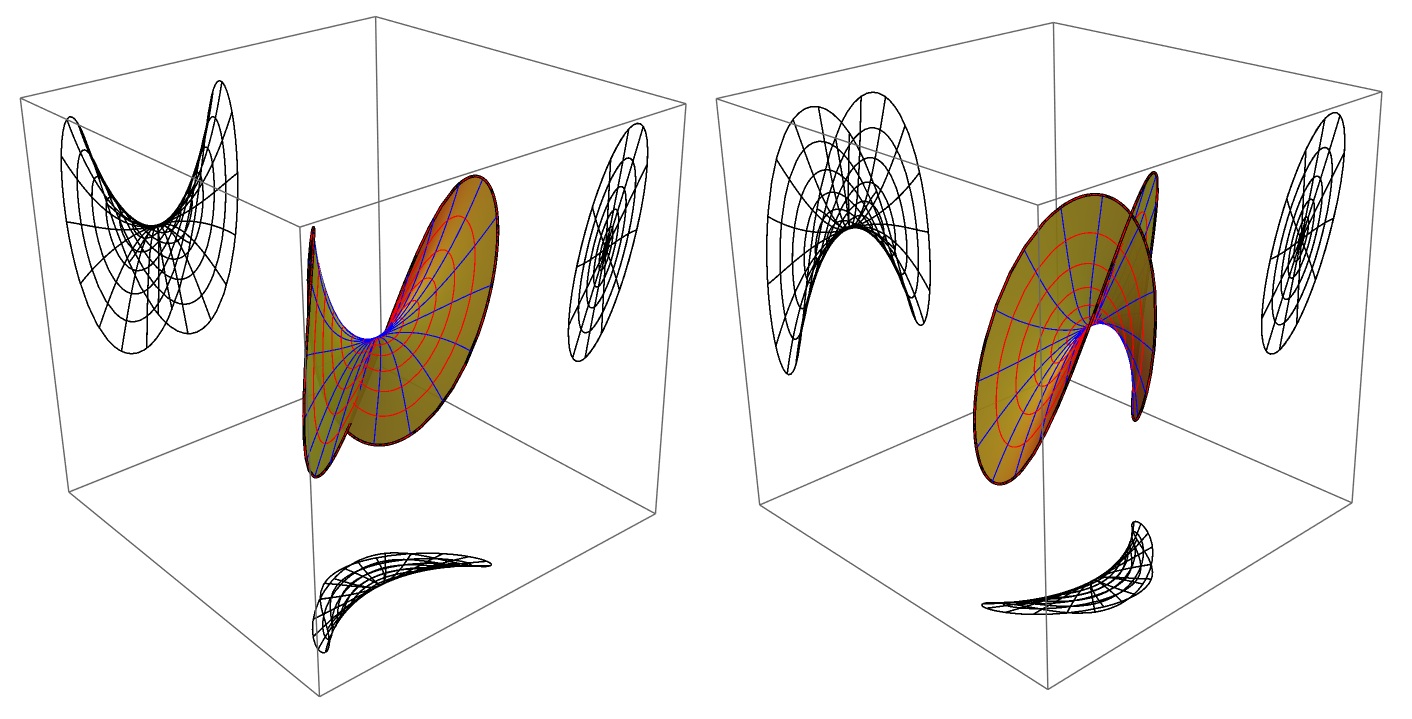}
   \caption{Projections of the Schwarz-type minimal surface $X(r,\theta)$ defined in \eqref{eq:weierstrass-Schwarz}, into the $X_1X_3X_4$-space (left)
and the $X_2X_3X_4$-space (right).}
    \label{fig:polar-proj-2}
\end{figure}

\medskip

\noindent
\emph{Third and fourth coordinates.}
We begin with the third coordinate function:
\[
\int \Phi_3
=
2\int
\frac{\omega}{\sqrt{\omega^8-14\omega^4+1}}\,d\omega
=
\int
\frac{dt}{\sqrt{t^4-14t^2+1}}.
\]
The quartic polynomial factorizes as
\[
t^4-14t^2+1=(t^2-\alpha^2)(t^2-\beta^2),
\qquad
\alpha=2+\sqrt{3},\quad
\beta=2-\sqrt{3}.
\]

Consequently $\int\Phi_3$ is expressible in terms of incomplete elliptic integrals of the first kind, after a standard normalization of the quartic radical.
Since $\Phi_4=\lambda\Phi_3$, the fourth primitive satisfies
\[
\int \Phi_4
=
\lambda\int \Phi_3.
\]

\medskip

\noindent
\emph{First and second coordinates.}
The remaining primitives involve integrals of the form
\[
\int
\frac{1}{\sqrt{\omega^8-14\omega^4+1}}\,d\omega,
\qquad
\int
\frac{\omega^2}{\sqrt{\omega^8-14\omega^4+1}}\,d\omega,
\]
which are Abelian integrals on the genus-three curve $\Sigma$.  Their evaluation can be organized using quotient symmetries of $\Sigma$ and, for suitable symmetry-adapted combinations, may be related to elliptic quotient integrals.  No explicit elliptic reduction for all first and second coordinate periods is needed for the symmetry theorem proved below.

\subsection{Geometric meaning of the elliptic quotient}

The branch configuration has nontrivial quotient symmetries, and the third and fourth coordinate periods are governed by the elliptic quotient displayed above.  This reflects the intrinsic algebraic symmetry of the underlying Schwarz polynomial and motivates the use of a P/D-type algebraic seed.

\begin{remark}
One may keep the primitives in hyperelliptic form on $\Sigma$ for conceptual clarity, or reduce them to elliptic standard forms when explicit period computations are desired.
For later period-closing problems, it is often advantageous to keep the $t$-variable quartic radical representation since it aligns well with classical TPMS computations.
\end{remark}

% =========================================================
% Classical Schwarz P/D/G data and the lift to R^4
% =========================================================
\section{Classical Schwarz P, D, and G Surfaces and a Codimension-Two Lift}\label{sec:PDG}

\subsection{Weierstrass data for the Schwarz P and D surfaces in $\mathbb{R}^3$}

Let $M$ be a conformal minimal immersion into $\mathbb{R}^3$.
On a simply connected domain, the Weierstrass-Enneper representation is
\begin{equation}\label{eq:WE-R3}
X(z)=\Re\int_{z_0}^z \Bigl(\tfrac12(1-g^2),\ \tfrac{i}{2}(1+g^2),\ g\Bigr)\,f\,dz,
\end{equation}
where $g$ is meromorphic and $f\,dz$ is a holomorphic 1-form.

A classical parametrization of the Schwarz P surface (in appropriate coordinates on its underlying branched cover) uses
\begin{equation}\label{eq:SchwarzPDdata}
g(z)=z,
\qquad
f(z)=\frac{1}{\sqrt{1-14z^4+z^8}}.
\end{equation}
The conjugate Schwarz D surface is obtained by the Bonnet rotation
\begin{equation}\label{eq:Ddata}
g(z)=z,
\qquad
f(z)=\frac{i}{\sqrt{1-14z^4+z^8}},
\end{equation}
and more generally the associate (Bonnet) family is
\begin{equation}\label{eq:associate-family}
f_\theta(z)=\frac{e^{i\theta}}{\sqrt{1-14z^4+z^8}},
\qquad
g(z)=z,
\qquad
\theta\in[0,\pi/2].
\end{equation}
In this normalization, $\theta=0$ corresponds to the P surface and $\theta=\pi/2$ to the D surface.
The gyroid (G) appears as a distinguished embedded member in the associate family at a specific Bonnet angle determined by period closure (see, for example, \cite{WeyhauptGyroid,GBrauckmannSurvey}).

\medskip

\noindent
\begin{remark} [period closure in $\mathbb{R}^3$].
Triply periodicity means that all real periods of the coordinate 1-forms in \eqref{eq:WE-R3} lie in a rank-3 lattice $\Lambda\subset\mathbb{R}^3$.
The period-map viewpoint and deformation theory for periodic minimal surfaces are developed in the TPMS literature; see, for example, \cite{PirolaSpin,GBrauckmannSurvey}.
\end{remark}

\subsection{A codimension-two lift of the P/D seed}

Motivated by \eqref{eq:SchwarzPDdata}--\eqref{eq:associate-family}, we propose a codimension-two lift built on the same algebraic curve:
\begin{equation}\label{eq:lift}
g_1(z)=z,
\qquad
g_2(z)=\lambda z,
\qquad
f_\theta(z)=\frac{e^{i\theta}}{\sqrt{1-14z^4+z^8}},
\end{equation}
with parameters $\lambda\in\mathbb{C}$ and $\theta\in[0,\pi/2]$.
When $\lambda=0$, the immersion lies in a 3-dimensional affine subspace and recovers the classical associate family.  More generally, if $\lambda\in\mathbb{R}$, then $\Phi_4=\lambda\Phi_3$ implies $X_4=\lambda X_3+\mathrm{constant}$, so the image is again contained in an affine $3$-space.  Thus a necessary condition for this lift to use the additional real target direction is $\operatorname{Im}\lambda\neq0$.

\begin{remark}
More general lifts can be built by allowing $g_1\neq g_2$ (beyond a scalar multiple) or by modifying the prefactor while preserving the nullity condition.
The present paper focuses on the clean symmetric lift \eqref{eq:lift}, which is already rich enough to reveal the core $\mathbb{R}^4$ period phenomena.
\end{remark}

% =========================================================
% Period problem section
% =========================================================
\section{The Real Period Map in $\mathbb{R}^4$ and Periodicity in $\mathbb{T}^4$}\label{sec:period}

\subsection{The real period map}

Let $\Sigma$ be a compact Riemann surface on which the holomorphic data are single-valued.
Let $\Phi=(\Phi_1,\Phi_2,\Phi_3,\Phi_4)$ be a holomorphic null $1$-form as above.
For any homology class $[\gamma]\in H_1(\Sigma;\mathbb{Z})$, define the real period
\begin{equation}\label{eq:real-period}
\mathcal{P}([\gamma])=\Re\int_\gamma \Phi\in\mathbb{R}^4.
\end{equation}

\begin{definition}[Period closure in $\mathbb{R}^4$]
We say that $X=\Re\int \Phi$ is \emph{single-valued} on $\Sigma$ if $\mathcal{P}([\gamma])=0$ for all $[\gamma]\in H_1(\Sigma;\mathbb{Z})$.
\end{definition}

In the periodic setting, one seeks a weaker condition.

\begin{definition}[Periodicity in a flat torus]
Let $\Lambda\subset\mathbb{R}^4$ be a rank-4 lattice and $\mathbb{T}^4=\mathbb{R}^4/\Lambda$.
We say that $X=\Re\int \Phi$ \emph{descends to a minimal immersion into $\mathbb{T}^4$} if
\begin{equation}\label{eq:period-closure}
\mathcal{P}\bigl(H_1(\Sigma;\mathbb{Z})\bigr)\subset \Lambda.
\end{equation}
\end{definition}

The viewpoint that periodic minimal surfaces are best understood as compact minimal immersions into flat tori is developed in the literature on minimal surfaces in $\mathbb{T}^n$, and in particular in $\mathbb{T}^4$; see, for example, \cite{ShodaJLMS,ShodaTrigonal,ShodaModuliArXiv}.

\subsection{Codimension-two features of the period system}

Even when the underlying algebraic curve is the same as in the classical P/D setting, the period system in $\mathbb{R}^4$ differs in two fundamental ways.

First, the target is $\mathbb{R}^4$, so the period vectors live in a four-dimensional real space; this changes the lattice-combinatorics of closure and introduces new possibilities for satisfying closure by allowing additional directions.

Second, the nullity condition couples the $1$-forms quadratically but still leaves substantial freedom in the pair $(g_1,g_2)$.
This flexibility suggests a natural deformation theory: for fixed underlying curve $\Sigma$, one can vary $\lambda$ and $\theta$ in \eqref{eq:lift} (and beyond) to search for solutions of \eqref{eq:period-closure}.

\begin{remark}
The parameter $\theta$ is classical (Bonnet angle) and rotates within the associate family in $\mathbb{R}^3$.
The parameter $\lambda$ is new: it activates the second normal direction and, in principle, may help solve higher-dimensional period constraints by introducing additional degrees of freedom not available in $\mathbb{R}^3$.
\end{remark}

% =========================================================
% Symmetry section
% =========================================================
\section{Symmetry of the Schwarz Curve and Induced Actions on Periods}\label{sec:symmetry}

\subsection{Algebraic symmetries of the hyperelliptic model}

Consider the hyperelliptic curve
\[
\Sigma:\quad y^2=\omega^8-14\omega^4+1.
\]
The polynomial $\omega^8-14\omega^4+1$ is invariant under $\omega\mapsto i\omega$ and $\omega\mapsto -\omega$, hence $\Sigma$ carries natural order-4 and order-2 automorphisms lifting these maps (with the appropriate action on $y$).
Moreover, the transformation $\omega\mapsto 1/\omega$ preserves the set of branch points and induces a further symmetry after multiplying by the appropriate factor on the $y$-coordinate.

These symmetries are the algebraic shadow of the reflection symmetries used in the classical Schwarz constructions.
Their key role in our program is that they induce linear relations among the periods of the holomorphic $1$-forms $\Phi_j$.

\subsection{Symmetry constraints on the period map}

Let $G$ be a finite subgroup of $\mathrm{Aut}(\Sigma)$ preserving the divisor of poles and zeros of the data.
Assume that the holomorphic null curve data are \emph{$G$-equivariant} in the sense that, for each $\sigma\in G$,
\[
\sigma^*\Phi = M_\sigma \Phi
\quad\text{for some constant matrix } M_\sigma\in U(4),
\]
so $\sigma$ acts linearly on the space of complex-valued $1$-forms.
Then the period map obeys
\[
\mathcal{P}(\sigma_*[\gamma])=\Re\int_{\sigma(\gamma)}\Phi
=\Re\int_\gamma \sigma^*\Phi
=\Re\int_\gamma (M_\sigma\Phi),
\]
which yields explicit linear relations between the real period vectors on $\gamma$ and on $\sigma(\gamma)$ once $M_\sigma$ is known.
Thus $G$-symmetry can force families of period vectors to be generated from a small set of fundamental ones.

\begin{remark}
The practical content of the Schwarz reflection philosophy is exactly this: choose data so that a finite symmetry group constrains the period map to satisfy explicit linear relations, thereby reducing period closure to finitely many real scalar conditions once a symmetry-adapted generating set of cycles has been chosen.
In codimension two, the same philosophy applies but the available complex linear symmetries $M_\sigma$ (and the induced real relations) can be richer, and the space of admissible data is larger.
\end{remark}

% =========================================================
% Strategy section
% =========================================================
\section{Scope of the Period Problem}\label{sec:strategy}

\subsection{A symmetry-adapted period basis}

For the curve $\Sigma$ defined by \eqref{eq:Schwarz-seed} (of genus $g=3$), one may choose a canonical homology basis adapted to branch cuts connecting pairs of branch points and to the symmetry $\omega\mapsto i\omega$.
In such a basis, many periods occur in symmetry-related quartets.
The codimension-two lift \eqref{eq:lift} then produces a period system depending on $(\lambda,\theta)$ with the schematic form
\[
\mathcal{P}_j(\lambda,\theta)\in\mathbb{R}^4,\qquad j=1,\dots,2g,
\]
and the symmetry theorem below shows exactly how the periods on a full orbit are recovered from a representative.  Solving the remaining scalar conditions for a prescribed lattice is a separate problem.

\subsection{Why codimension two can help, but does not trivialize the problem}

It is tempting to view $\mathbb{R}^4$ as offering ``more room'' and therefore making closure easier.
However, the real period problem remains rigid for two reasons:\\
A). The nullity constraint couples the $1$-forms quadratically and forbids arbitrary choices of $\Phi$.\\
B). Period closure is a global condition on $H_1(\Sigma;\mathbb{Z})$ and interacts strongly with symmetry and branching.\\

What codimension two offers is not automatic solvability, but rather new deformation directions (such as $\lambda$ in \eqref{eq:lift}) that can be used to search for solutions of a reduced period system.

\subsection{What is not proved here}

The preceding reduction does not prove the existence of parameters for which the periods close in a prescribed rank-four lattice.  It also does not establish completeness, embeddedness, or a reflection extension of the resulting immersion.  These questions require additional global arguments and are deliberately not inferred from the local null-curve formulae or from the symmetry relations alone.

% =========================================================
% Topology remarks
% =========================================================
\section{Limitations and Relation to Reflection--Plateau Constructions}\label{sec:topology}

The genus of the curve agrees with that of the standard compact quotients of the Schwarz P/D surfaces, but this agreement alone has no embeddedness consequence.  In particular, extra ambient dimension does not establish sheet separation.  Periodic minimal immersions into flat tori remain subject to global period and regularity constraints.

This point distinguishes the present approach from the construction of Grosse-Brauckmann and Kürsten \cite{GrosseBrauckmannKuersten}.  They begin with a Plateau disk spanning a cubical Jordan curve and use successive Schwarz reflections to obtain complete $n$-periodic surfaces; their analysis supplies embedded examples in $\mathbb{R}^4$.  Here we instead prescribe holomorphic null data on a hyperelliptic curve and establish only the resulting equivariance and period identities.  The two approaches are complementary, but neither the Plateau argument nor its embeddedness conclusion is part of the present proof.

% =========================================================
% NEW SECTION: Symmetry reduction theorem (novel result)
% =========================================================
\section{Symmetry Reduction of the Period System for Schwarz-Type Lifts}\label{sec:newresult}

In this section we give a symmetry-driven reduction principle for the real period map associated to the Schwarz-type lift from Section~\ref{sec:seed}.  The result is structural: it does not solve the period problem, but it shows that (i) the period vectors satisfy explicit symmetry relations, and (ii) the number of independent real period \emph{constraints} can be reduced to a small subset once a symmetry-adapted generating set of cycles is chosen.

\subsection{An order-four automorphism of the Schwarz curve}

Consider the hyperelliptic curve
\[
\Sigma:\quad y^2=\omega^8-14\omega^4+1.
\]
Since $(i\omega)^8-14(i\omega)^4+1=\omega^8-14\omega^4+1$, the map
\[
\sigma:\Sigma\to\Sigma,\qquad \sigma(\omega,y)=(i\omega,y),
\]
is a holomorphic automorphism of order four.

Let $\Phi=(\Phi_1,\Phi_2,\Phi_3,\Phi_4)$ be the holomorphic $1$-forms defined by \eqref{eq:WE-R4} with the Schwarz-type seed
\[
f(\omega)=\frac{2}{\sqrt{\omega^8-14\omega^4+1}},\qquad g_1(\omega)=\omega,\qquad g_2(\omega)=\lambda\omega,
\]
so that $\Phi_j=\phi_j(\omega)\,d\omega$ with $\phi_j$ as in Section~\ref{sec:seed}.

\begin{lemma}[Equivariance of the holomorphic $1$-forms]\label{lem:equivariance_correct}

Recall that $\Phi_j=\phi_j(\omega)\,d\omega$; in what follows, $\sigma^*\Phi_j$ also includes the pullback of $d\omega$.

Under $\sigma(\omega)=i\omega$ the holomorphic $1$-forms satisfy
\begin{equation}\label{eq:sigma_on_Phi}
\sigma^*\Phi_1=\Phi_2,\qquad
\sigma^*\Phi_2=-\Phi_1,\qquad
\sigma^*\Phi_3=-\Phi_3,\qquad
\sigma^*\Phi_4=-\Phi_4.
\end{equation}
Equivalently, with $\Phi$ viewed as a column vector of complex-valued $1$-forms, one has
\[
\sigma^*\Phi = M\,\Phi,
\]
where
\[
M=
\begin{pmatrix}
0 & -1 & 0 & 0\\
1 &  0 & 0 & 0\\
0 & 0 & -1 & 0\\
0 & 0 & 0 & -1
\end{pmatrix}
\in U(4).
\]
\end{lemma}

\begin{proof}
Because $P(i\omega)=P(\omega)$, one has $f(i\omega)=f(\omega)$. Moreover,
\[
g_1(i\omega)=i g_1(\omega),\qquad g_2(i\omega)=i g_2(\omega),
\]
so $g_1^2+g_2^2$ changes sign under $\omega\mapsto i\omega$.
At the level of coefficient functions, this yields
\[
\phi_1(i\omega)=\tfrac12 f(\omega)\bigl(1+g_1^2+g_2^2\bigr)=\frac{1}{i}\,\phi_2(\omega),
\qquad
\phi_2(i\omega)=\frac{i}{2}f(\omega)\bigl(1-g_1^2-g_2^2\bigr)=i\,\phi_1(\omega),
\]
and $\phi_3(i\omega)=i\,\phi_3(\omega)$, $\phi_4(i\omega)=i\,\phi_4(\omega)$.
Since $\sigma^*(d\omega)=d(i\omega)=i\,d\omega$, one obtains
\[
\sigma^*\Phi_1=\phi_1(i\omega)\,i\,d\omega=\phi_2(\omega)\,d\omega=\Phi_2,
\qquad
\sigma^*\Phi_2=\phi_2(i\omega)\,i\,d\omega=-\phi_1(\omega)\,d\omega=-\Phi_1,
\]
and similarly $\sigma^*\Phi_3=-\Phi_3$, $\sigma^*\Phi_4=-\Phi_4$.
\end{proof}

\subsection{Consequences for the real period map}

Define the real period map
\[
\mathcal{P}:H_1(\Sigma;\mathbb{Z})\to\mathbb{R}^4,\qquad
\mathcal{P}([\gamma])=\Re\int_\gamma (\Phi_1,\Phi_2,\Phi_3,\Phi_4).
\]

\begin{theorem}[Symmetry relations among real period vectors]\label{thm:symmetryreduction_correct}
For every $[\gamma]\in H_1(\Sigma;\mathbb{Z})$ one has the period identities
\begin{equation}\label{eq:period_identities}
\begin{aligned}
\Re\int_{\sigma(\gamma)}\Phi_1 &= \Re\int_\gamma \Phi_2,\\
\Re\int_{\sigma(\gamma)}\Phi_2 &= -\,\Re\int_\gamma \Phi_1,\\
\Re\int_{\sigma(\gamma)}\Phi_3 &= -\,\Re\int_\gamma \Phi_3,\\
\Re\int_{\sigma(\gamma)}\Phi_4 &= -\,\Re\int_\gamma \Phi_4.
\end{aligned}
\end{equation}
In particular, if a set of homology classes $\Gamma\subset H_1(\Sigma;\mathbb{Z})$ generates $H_1(\Sigma;\mathbb{Z})$ and is closed under the $\sigma$-action, then the real period vectors $\mathcal{P}([\gamma])$ for all $[\gamma]\in\Gamma$ are determined by the real periods on any choice of orbit representatives in $\Gamma$, via \eqref{eq:period_identities} and its iterates.
\end{theorem}

\begin{proof}
For any cycle $\gamma$,
\[
\int_{\sigma(\gamma)}\Phi_j=\int_\gamma \sigma^*\Phi_j.
\]
Applying Lemma~\ref{lem:equivariance_correct} and taking real parts yields \eqref{eq:period_identities}.
If $\Gamma$ is $\sigma$-stable, then the $\sigma$-iterates of orbit representatives determine the periods on all cycles in $\Gamma$, hence the claim.
\end{proof}

\begin{corollary}[Symmetry-reduced period-closing constraints]\label{cor:periodreduction_correct}
Let $R\in O(4)$ be the real matrix induced by $M$ in \eqref{eq:sigma_on_Phi}; explicitly,
\[
R(x_1,x_2,x_3,x_4)=(x_2,-x_1,-x_3,-x_4).
\]
If the lattice $\Lambda\subset\mathbb{R}^4$ is $R$-invariant, then for the Schwarz-type family (including the lift \eqref{eq:lift}), the torus-periodicity condition
\[
\mathcal{P}\bigl(H_1(\Sigma;\mathbb{Z})\bigr)\subset\Lambda\subset\mathbb{R}^4
\]
can be checked on a symmetry-reduced set of cycles consisting of one representative from each $\sigma$-orbit in a $\sigma$-stable generating set of $H_1(\Sigma;\mathbb{Z})$ (with the corresponding relations among their periods given by \eqref{eq:period_identities}).
\end{corollary}

\begin{proof}
By Theorem~\ref{thm:symmetryreduction_correct}, the real periods on a $\sigma$-stable generating set are obtained from periods on orbit representatives by iterating $R$.  Since $R\Lambda=\Lambda$, membership of a representative period in $\Lambda$ is equivalent to membership of every period in its orbit.  This proves the assertion.
\end{proof}

\begin{remark}[Conceptual significance]
The identities \eqref{eq:period_identities} are the codimension-two analogue of the classical Schwarz reflection reduction: symmetry of the algebraic data forces explicit linear relations among periods.
This justifies, at a structural level, the strategy of imposing symmetry first and then solving a reduced finite system of real closure conditions in the parameters (such as $(\lambda,\theta)$).
\end{remark}

\section*{Appendix A: Standard Elliptic Normal Forms for the Quartic Radical} \label{sec:appendixA}

For period computations it is often convenient to use explicit standard normal forms.
Consider
\[
I(t)=\int \frac{dt}{\sqrt{(t^2-\alpha^2)(t^2-\beta^2)}},
\qquad \alpha>\beta>0.
\]
A classical substitution (one of several equivalent choices) reduces $I(t)$ to an incomplete elliptic integral of the first kind, with modulus $k=\beta/\alpha$.
More generally, integrals of the form
\[
\int \frac{dt}{\sqrt{(t^2-\alpha^2)(t^2-\beta^2)}},\quad
\int \frac{t^2\,dt}{\sqrt{(t^2-\alpha^2)(t^2-\beta^2)}},\quad
\int \frac{dt}{t^2\sqrt{(t^2-\alpha^2)(t^2-\beta^2)}}
\]
can be expressed as linear combinations of $F(\cdot|k)$ and $E(\cdot|k)$.
A convenient reference for explicit tables of such reductions is \cite{ByrdFriedman}.

\begin{remark}
The elliptic normal form applies directly to the third and fourth coordinate periods.  The other coordinate periods remain Abelian integrals on $\Sigma$ unless one specifies an additional quotient or symmetry reduction.  We therefore do not use this appendix to claim a complete closed-form solution of the period problem.
\end{remark}

\section*{Appendix B: Analytic and Geometric Interpretation of the Schwarz-Type Seed} \label{sec:appendixB}

In this appendix we clarify the analytic structure of the Schwarz-type seed
\[
f(\omega)=\frac{2}{\sqrt{\omega^8-14\omega^4+1}},
\]
and explain how its algebraic branch locus enters the Weierstrass $1$-forms used in the paper.

\subsection*{B.1 Branch points of the hyperelliptic curve}

Let
\[
P(\omega)=\omega^8-14\omega^4+1,
\qquad
\Sigma=\{(\omega,y)\in\mathbb{C}^2:\ y^2=P(\omega)\}.
\]
The zeros of $P(\omega)$ are precisely the branch points of the two-sheeted hyperelliptic covering
$\pi:\Sigma\to\mathbb{C}$ given by $\pi(\omega,y)=\omega$.

To locate them, set $t=\omega^4$. Then $P(\omega)=0$ is equivalent to
\[
t^2-14t+1=0,
\qquad
t=7\pm4\sqrt{3}.
\]
Since
\[
7\pm4\sqrt{3}=(2\pm\sqrt{3})^2,
\]
we obtain
\[
\omega^4=(2\pm\sqrt{3})^2,
\]
and hence the eight branch points are
\[
\omega\in
\Big\{
\pm\sqrt{2+\sqrt{3}},\ \pm i\sqrt{2+\sqrt{3}},\
\pm\sqrt{2-\sqrt{3}},\ \pm i\sqrt{2-\sqrt{3}}
\Big\}.
\]
Equivalently, the branch points lie on the two circles
\[
|\omega|=\sqrt{2-\sqrt{3}}\approx0.5176,
\qquad
|\omega|=\sqrt{2+\sqrt{3}}\approx1.9319.
\]

Since $\deg P=8$ and all roots are simple, the compactification of $\Sigma$ is a hyperelliptic
Riemann surface of genus
\[
g=\frac{8-2}{2}=3.
\]

\subsection*{B.2 Two-sheeted structure and local coordinates at branch points}

The square root in the definition of $f(\omega)$ becomes single-valued on $\Sigma$ by setting
$y=\sqrt{P(\omega)}$. Under analytic continuation around any branch point, the local determination
of $y$ changes sign:
\[
y\longmapsto -y.
\]
Equivalently, $\Sigma$ is the natural domain on which $y$ is holomorphic.

A crucial point for the Weierstrass construction is that although $f=2/y$ has poles at points where $y=0$,
the \emph{holomorphic $1$-forms} used to build the immersion are of the form
\[
\Phi_j=\phi_j(\omega)\,d\omega,
\]
and these can be regular at $y=0$ because $d\omega$ vanishes in a true local coordinate on $\Sigma$.

Indeed, near a branch point $\omega_0$ one may choose a local coordinate $s$ on $\Sigma$ with
\[
\omega-\omega_0 = s^2,\qquad y = s\cdot h(s),
\]
where $h(0)\neq 0$. Then
\[
d\omega = 2s\,ds.
\]
Thus $f\,d\omega = (2/y)\,d\omega$ has the local behavior
\[
\frac{2}{y}\,d\omega \sim \frac{2}{s}\cdot (2s\,ds)=4\,ds,
\]
and hence is holomorphic. Since the remaining factors in $\phi_j$ are holomorphic, each $\Phi_j$ extends holomorphically across $y=0$.

\subsection*{B.3 Holomorphic nature of the Weierstrass $1$-forms on $\Sigma$}

In the main text we use
\[
f(\omega)=\frac{2}{y},\qquad g_1(\omega)=\omega,\qquad g_2(\omega)=\lambda\omega,
\]
and define $\Phi=(\Phi_1,\Phi_2,\Phi_3,\Phi_4)$ by
\[
\Phi_1=\tfrac12 f(1-g_1^2-g_2^2)\,d\omega,\quad
\Phi_2=\tfrac{i}{2}f(1+g_1^2+g_2^2)\,d\omega,\quad
\Phi_3=fg_1\,d\omega,\quad
\Phi_4=fg_2\,d\omega.
\]

As explained in B.2, although $f$ itself is meromorphic on $\Sigma$, the products defining $\Phi_j$
extend holomorphically across the branch locus.

Moreover, one checks similarly at the points at infinity on the compactification of $\Sigma$
(for example by using $\zeta=1/\omega$ as a parameter) that the $1$-forms $\Phi_j$ extend meromorphically there.
Consequently, the Schwarz-type data define a meromorphic null $1$-form $\Phi$ on the compact genus-$3$ surface $\Sigma$, which is holomorphic on the affine hyperelliptic curve away from the points at infinity.

From the standpoint of the immersion
\[
X=\Re\int \Phi,
\]
this is the cleanest analytic situation on the affine curve: the $1$-forms have no singularities away from the points at infinity, and the only global obstruction to single-valuedness is the real period map discussed in Section~\ref{sec:period}.

\subsection*{B.4 Induced metric: correct structural dependence and coordinate caution}

The induced metric is
\[
ds^2=|\Phi_1|^2+|\Phi_2|^2+|\Phi_3|^2+|\Phi_4|^2.
\]
In a local coordinate $\omega$ for which $\Phi_j=\phi_j(\omega)\,d\omega$, one has
\[
ds^2=
\bigl(|\phi_1|^2+|\phi_2|^2+|\phi_3|^2+|\phi_4|^2\bigr)\,|d\omega|^2
=
\frac{|f(\omega)|^2}{2}\bigl(1+|g_1(\omega)|^2+|g_2(\omega)|^2\bigr)^2|d\omega|^2.
\]
However, at a branch point of $\pi:\Sigma\to\mathbb{C}$, the projection coordinate $\omega$ is \emph{not}
a true local coordinate on $\Sigma$. In a genuine local coordinate $s$ on $\Sigma$ (with $\omega-\omega_0=s^2$),
the forms $\Phi_j$ are holomorphic and the metric is finite and well defined.
Thus any apparent blow-up of the coefficient expression in the $\omega$-plane is a coordinate artifact of viewing $\Sigma$
through the branched projection.

\begin{figure}[t!]
    \centering
    \arxivgraphic{width=0.5\linewidth}{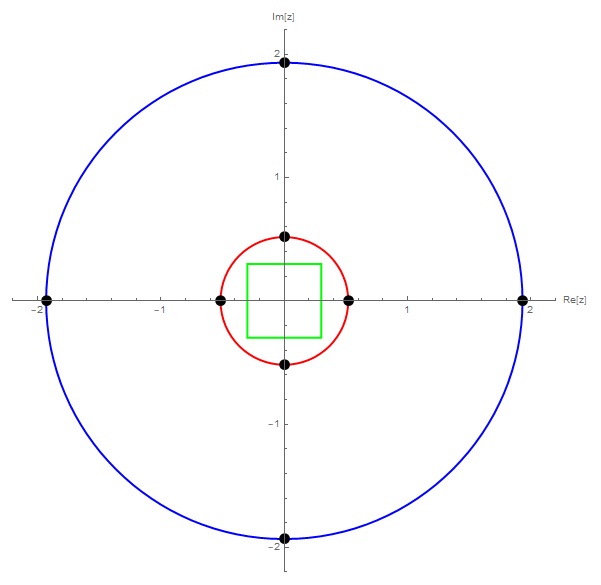}
    \caption{
    The branch locus of $P(\omega)=\omega^8-14\omega^4+1$ consists of eight points on the circles 
$|\omega|=\sqrt{2\pm\sqrt{3}}$, where the sheets of $y=\sqrt{P(\omega)}$ interchange.
    }
    \label{fig:singularity}
\end{figure}

\subsection*{B.5 Interpretation for visualizations}

When plotting in the $\omega$-plane, the eight branch points appear as four points on each of the
two circles
\[
|\omega|=\sqrt{2-\sqrt{3}},
\qquad
|\omega|=\sqrt{2+\sqrt{3}}.
\]
See Figure~\ref{fig:singularity}. A choice of branch cuts in the $\omega$-plane corresponds to choosing a simply connected domain on which a single branch of $y=\sqrt{P(\omega)}$ is selected; moving across a cut interchanges the sheets (i.e.\ sends $y\mapsto -y$). 

From the viewpoint of numerical evaluation, one typically avoids neighborhoods of the branch locus in the \emph{projection plane} to maintain numerical stability; nonetheless, on the surface $\Sigma$ itself the holomorphic forms remain regular.

\subsection*{B.6 Conceptual summary}

The Schwarz polynomial $P(\omega)=\omega^8-14\omega^4+1$ determines a genus-$3$ hyperelliptic
curve $\Sigma:\ y^2=P(\omega)$ with eight branch points arranged on two concentric circles in the
$\omega$-plane. The seed $f(\omega)=2/\sqrt{P(\omega)}$ is naturally interpreted as $f=2/y$ on $\Sigma$.
Although $f$ is meromorphic, the Weierstrass objects that govern the immersion are the holomorphic $1$-forms
$\Phi_j=\phi_j(\omega)\,d\omega$, and these extend holomorphically across the branch points, and meromorphically at the points at infinity on the compactification.
Hence, the global analytic structure emphasized throughout the paper is controlled by the real period map,
not by singularities of the differential data.

\section*{Declarations}

\noindent\textbf{Funding} \\
This research received no external funding.

\vspace{0.6em}
\noindent\textbf{Conflict of interest} \\
The authors declare that they have no conflict of interest.

\vspace{0.6em}
\noindent\textbf{Data availability} \\
No datasets were generated or analyzed during the current study.

\vspace{0.6em}
\noindent\textbf{Author contributions} \\
All authors contributed equally to the conceptualization, analysis, and writing of the manuscript. All authors read and approved the final version.

%\begin{appendices}

%\section{Section title of first appendix}\label{secA1}

%An appendix contains supplementary information that is not an essential part of the text itself but which may be helpful in providing a more comprehensive understanding of the research problem or it is information that is too cumbersome to be included in the body of the paper.

%%=============================================%%
%% For submissions to Nature Portfolio Journals %%
%% please use the heading ``Extended Data''.   %%
%%=============================================%%

%%=============================================================%%
%% Sample for another appendix section			       %%
%%=============================================================%%

%% \section{Example of another appendix section}\label{secA2}%
%% Appendices may be used for helpful, supporting or essential material that would otherwise 
%% clutter, break up or be distracting to the text. Appendices can consist of sections, figures, 
%% tables and equations etc.

%\end{appendices}

%%===========================================================================================%%
%% If you are submitting to one of the Nature Portfolio journals, using the eJP submission   %%
%% system, please include the references within the manuscript file itself. You may do this  %%
%% by copying the reference list from your .bbl file, paste it into the main manuscript .tex %%
%% file, and delete the associated \verb+\bibliography+ commands.                            %%
%%===========================================================================================%%

\bigskip
\noindent
\textbf{Erhan G\"uler}\\
Department of Mathematics and Statistics, Texas Tech University,\\
Lubbock, TX 79409, USA\\
\textit{Email:} \texttt{eguler@ttu.edu}

\bigskip
\noindent
\textbf{Magdalena Toda}\\
Department of Mathematics and Statistics, Texas Tech University,\\
Lubbock, TX 79409, USA\\
\textit{Email:} \texttt{magda.toda@ttu.edu}


\begin{thebibliography}{99}

\bibitem{AiyamaAkutagawa}
R.~Aiyama and K.~Akutagawa,
Kenmotsu type representation formulae for surfaces with prescribed mean curvature in the 
$3$-sphere,
\textit{Tohoku Math. J.} \textbf{52} (2000), 95--105.
doi: \href{https://doi.org/10.2748/tmj/1178224660}{10.2748/tmj/1178224660}.

\bibitem{BoltonPeditWoodward}
J.~Bolton, F.~Pedit, and L.~M.~Woodward,
Minimal surfaces and the affine Toda field equations,
\textit{J. Reine Angew. Math.} \textbf{459} (1995), 119--150.
doi: \href{https://doi.org/10.1515/crll.1995.459.119}{10.1515/crll.1995.459.119}.

\bibitem{ByrdFriedman}
P.~F.~Byrd and M.~D.~Friedman,
\textit{Handbook of Elliptic Integrals for Engineers and Scientists},
2nd ed., Springer, 1971.
doi: \href{https://doi.org/10.1007/978-3-642-65138-0}{10.1007/978-3-642-65138-0}.

\bibitem{GBrauckmannSurvey}
K.~Grosse-Brauckmann,
Triply periodic minimal and constant mean curvature surfaces,
\textit{Interface Focus} \textbf{2} (2012), 582--588.
doi: \href{https://doi.org/10.1098/rsfs.2011.0096}{10.1098/rsfs.2011.0096}.

\bibitem{GrosseBrauckmannKuersten}
K.~Grosse-Brauckmann and S.~Kürsten,
Construction of embedded periodic surfaces in $\mathbb{R}^n$,
arXiv:1707.09176 (2017).

\bibitem{Guest}
M.~A.~Guest,
\textit{Harmonic Maps, Loop Groups, and Integrable Systems},
Cambridge University Press, Cambridge, 1997.
doi: \href{https://doi.org/10.1017/CBO9781139174848}{10.1017/CBO9781139174848}.

\bibitem{TodaGulerRn}
E.~G\"uler and M.~Toda,
Weierstrass-type constructions, variational analysis and integral-free minimal immersions in $\mathbb{R}^n$,
\textit{Electron. Res. Arch.} \textbf{34}(3) (2026), 1885--1899.
doi: \href{https://doi.org/10.3934/era.2026084}{10.3934/era.2026084}.

\bibitem{HoffmanMeeks}
K.~Hoffman and W.~H.~Meeks III,
A complete embedded minimal surface in $\mathbb{R}^3$ with genus one and three ends,
\textit{J. Differential Geom.} \textbf{21} (1985), 109--127.
doi: \href{https://doi.org/10.4310/jdg/1214439467}{10.4310/jdg/1214439467}.

\bibitem{Kobayashi}
S.~Kobayashi,
\textit{Differential Geometry of Complex Vector Bundles},
Princeton University Press, Princeton, NJ, 1987.
doi: \href{https://doi.org/10.1515/9781400858682}{10.1515/9781400858682}.

\bibitem{LopezRos}
F.~J.~L\'opez and A.~Ros,
On embedded complete minimal surfaces of genus zero,
\textit{J. Differential Geom.} \textbf{33} (1991), 293--300.
doi: \href{https://doi.org/10.4310/jdg/1214446040}{10.4310/jdg/1214446040}.

\bibitem{Nitsche}
J.~C.~C.~Nitsche,
\textit{Lectures on Minimal Surfaces}, Vol.~1,
Cambridge University Press, Cambridge, 1989.

\bibitem{Osserman}
R.~Osserman,
\textit{A Survey of Minimal Surfaces},
Dover Publications, New York, 1986.

\bibitem{PirolaSpin}
G.~P.~Pirola,
The infinitesimal variation of the spin abelian differentials and periodic minimal surfaces,
\textit{Comm. Anal. Geom.} \textbf{6} (1998), 393--426.
doi: \href{https://doi.org/10.4310/CAG.1998.v6.n3.a1}{10.4310/CAG.1998.v6.n3.a1}.

\bibitem{Schwarz}
H.~A.~Schwarz,
\textit{Gesammelte mathematische Abhandlungen},
Springer, Berlin, 1890.

\bibitem{ShodaJLMS}
T.~Shoda,
New components of the moduli space of minimal surfaces in 4-dimensional flat tori,
\textit{J. London Math. Soc.} (2) \textbf{70} (2004), 797--816.
doi: \href{https://doi.org/10.1112/S0024610704005903}{10.1112/S0024610704005903}.

\bibitem{ShodaTrigonal}
T.~Shoda,
Trigonal minimal surfaces in flat tori,
\textit{Pacific J. Math.} \textbf{232} (2007), 401--422.
doi: \href{https://doi.org/10.2140/pjm.2007.232.401}{10.2140/pjm.2007.232.401}.

\bibitem{ShodaModuliArXiv}
T.~Shoda,
New components of the moduli space of minimal surfaces in 4-dimensional flat tori,
arXiv:math/0408052.

\bibitem{SmallR4}
A.~Small,
Algebraic minimal surfaces in $\mathbb{R}^4$,
\textit{Math. Scand.} \textbf{94} (2004), 109--124.
doi: \href{https://doi.org/10.7146/math.scand.a-14432}{10.7146/math.scand.a-14432}.

\bibitem{TodaGulerR4}
M.~Toda and E.~G\"uler,
Generalized Weierstrass--Enneper representation for minimal surfaces in $\mathbb{R}^4$,
\textit{AIMS Mathematics} \textbf{10}(9) (2025), 22406--22420.
doi: \href{https://doi.org/10.3934/math.2025997}{10.3934/math.2025997}.

\bibitem{WeyhauptGyroid}
A.~G.~Weyhaupt,
Deformations of the gyroid and Lidinoid minimal surfaces,
\textit{Pacific J. Math.} \textbf{235} (2008), 137--171.
doi: \href{https://doi.org/10.2140/pjm.2008.235.137}{10.2140/pjm.2008.235.137}.





\end{thebibliography}
\end{document}